\documentclass[12pt]{article}
\usepackage{amssymb,amsmath,amsfonts,amsthm,amsxtra,setspace}

\newcommand{\op}{\operatorname}

\newcommand{\mc}{\mathcal}

\newcommand{\set}[1]{\{#1\}}
\newcommand{\mf}[1]{\mathfrak{#1}}
\newcommand{\abs}[1]{\left| #1 \right|}

\newtheorem{prob}{Problem}

\newtheorem{theorem}{Theorem}

\newenvironment{defn}{\begin{trivlist} \item[] {\bf Definition.}\it}{\hspace*{0pt}\end{trivlist}}

\newsavebox{\Prfref}

\newsavebox{\prfref}

\newtheoremstyle{ref}
{\topsep}	
{\topsep}	
{\it}
{}
{}
{}
{ }
{\thmname{{\bfseries#1}}\thmnumber{ \textbf{#2\thmnote{\rm #3}\textbf .}}}
\newcommand{\ignore}[1]{}
\theoremstyle{ref}
\newtheorem{lem}[theorem]{Lemma}
\newtheorem{cor}[theorem]{Corollary}
\newtheorem{thm}[theorem]{Theorem}

\begin{document}
\title{Some observations on the Baireness of $C_k(X)$ for a locally compact space $X$}
\author{Franklin D. Tall}

\footnotetext[1]{Research supported by NSERC grant A-7354.\vspace*{2pt}}
\date{\today}
\maketitle

\begin{abstract}
We prove some consistency results concerning the Moving Off Property for locally compact spaces, and thus the question of whether their function spaces are Baire.
\end{abstract}

\renewcommand{\thefootnote}{}
\footnote{\parbox[1.8em]{\linewidth}{AMS $2000$ MSC: Primary 03E35, 54A35, 54C35, 54D45, 54E52; secondary 03E55, 03E65, 54D20, 54G20}\vspace*{5pt}}

\renewcommand{\thefootnote}{}

\footnote
{\parbox[1.8em]{\linewidth}{Key words and phrases: Baire, Moving Off Property, $C_k(X)$, locally compact, perfectly normal, PFA$(S)[S]$.}}

\section{Introduction}

The Moving Off Property was introduced in \cite{Gruenhage} to characterize when $C_k(X)$ satisfies the Baire Category Theorem, for $q$-spaces $X$.  Here we shall only be concerned with locally compact spaces (which are $q$), and so won't define $q$.  We shall assume all spaces are Hausdorff.

\begin{defn}
A \emph{moving off collection} for a space $X$ is a collection $\mc{K}$ of non-empty compact sets such that for each compact $L$, there is a $K \in \mc{K}$ disjoint from $L$.  A space satisfies the Moving Off Property (MOP) if each moving off collection includes an infinite subcollection with a discrete open expansion.
\end{defn}

\begin{thm}[{~\cite{Gruenhage}}]
A locally compact space $X$ satisfies the MOP if and only if $C_k(X)$ is Baire, i.e., satisfies the Baire Category Theorem.
\end{thm}

There is a less onerous equivalent of the MOP for locally compact spaces:

\begin{lem}[{~\cite{G}}]
Let $X$ be a locally compact space.  Then $X$ has the MOP if and only if every moving off collection for $X$ includes an infinite discrete subcollection.
\end{lem}
We give a proof for the benefit of readers who are not topologists.

\begin{proof}
Let $\mc{K}$ be a moving off collection for $X$.  By local compactness, each $K \in \mc{K}$ can be fattened to an open set with compact closure.  Let $\mc{K}'$ be the collection of all compact closures of open sets around members of $\mc{K}$.  Then $\mc{K}'$ is moving off.  For let $C$ be a compact subset of $X$.  There is a $K \in \mc{K}$ disjoint from $C$.  By regularity and local compactness, there is an open $U \supseteq K$ with compact closure $\overline{U}$ disjoint from $C$.  Then $\overline{U} \in \mc{K}'$.  Since we have established that $\mc{K}'$ is moving off, by hypothesis it includes an infinite discrete collection $\{\overline{U}_n\}_{n < \omega}$.  But each $U_n$ included some $K_n \in \mc{K}$.  Then $\{K_n\}_{n < \omega}$ is discrete and has the discrete open expansion $\{U_n\}_{n < \omega}$.
\end{proof}

In \cite{LT1}, \cite{LT2}, and \cite{Tallb}, assuming the existence of a supercompact cardinal, a model of set theory is constructed, which we shall refer to as a \emph{model of PFA$(S)[S]$}.  We refer the reader to those papers for a discussion of what PFA$(S)[S]$ is.  In these papers various propositions concerning locally compact normal spaces are established in this model.  We shall use:

\begin{lem}[{~\cite{LT2}}]\label{lem1}
In this model, locally compact hereditarily normal spaces which do not include a perfect pre-image of $\omega_1$ are paracompact.
\end{lem}

\begin{cor}[{~\cite{LT1}}]\label{cor2}
In this model, locally compact, perfectly normal spaces are paracompact.
\end{cor}

\begin{lem}[{~\cite{Tallb}}]\label{lem3}
In this model, locally compact normal spaces with Lindel\"of number $\leq \aleph_1$ which do not include a perfect pre-image of $\omega_1$ are paracompact.
\end{lem}

Let us also quote several useful results concerning the MOP.

\begin{lem}[{~\cite{MN, Gruenhage}}]\label{lem4}
Countably compact spaces satisfying the MOP are compact.
\end{lem}

\begin{lem}[{~\cite{MN, Gruenhage}}]\label{lem5}
First countable spaces satisfying the MOP are locally compact.
\end{lem}

\begin{lem}[{~\cite{MN, Gruenhage}}]\label{lem6}
Locally compact, paracompact spaces satisfy the MOP.
\end{lem}
A stronger result is in Lemma \ref{lemAlphaFavour} below

\section{Locally compact, perfectly normal spaces and the MOP}

Marion Scheepers asked us whether locally compact, perfectly normal spaces satisfy the MOP, and whether - if they do - they are paracompact.  Here are the answers, modulo a supercompact cardinal.

\begin{thm}\label{thm7}
There is a model of PFA$(S)[S]$ in which locally compact, perfectly normal spaces are paracompact and hence satisfy the MOP.
\end{thm}

\begin{thm}\label{thm8}
There is a model in which there is a locally compact, perfectly normal space which does not satisfy the MOP.
\end{thm}

\begin{proof}[Proofs]
Theorem \ref{thm7} follows from Corollary \ref{cor2} plus Lemma \ref{lem6}.  Theorem \ref{thm8} follows from Lemma \ref{lem4}, since \textit{Ostaszewski's space} \cite{O}, constructed from $\diamondsuit$, is locally compact, perfectly normal, countably compact, but not compact.
\end{proof}

For the other question, obviously Corollary \ref{cor2} answers it one way; for the other, we quote:

\begin{lem}[{~\cite{Ma}}]\label{lem9}
MA$_{\omega_1}$ implies there is a locally compact perfectly normal space with the MOP which is not paracompact.
\end{lem}

\section{Counterexamples}
Although the question of whether locally compact normal spaces with the MOP are paracompact has not been answered in ZFC, there are a number of consistent counterexamples which repurpose spaces familiar to normal Moore space fans.  a)-f) are not collectionwise Hausdorff, hence not paracompact.  Each is normal in some model.
\begin{enumerate}
\item[a)]{\cite{Ma} The Cantor tree on a set of reals of size $\aleph_1$ is normal and has the MOP under MA$_{\omega_1}$.
}
\end{enumerate}
\begin{defn}
A ladder system $\{\lambda_\alpha\}_{\alpha \in S}$, where $S$ is a subset of some ordinal $\lambda$, is a set of sequences, where each $\lambda_\alpha$ is strictly increasing, converges to $\alpha$, and has range disjoint from $S$. The corresponding ladder system space on $S \cup \bigcup\set{\op{range}\lambda_\alpha : \alpha \in S}$ has the points in each $\op{range} \lambda_\alpha$ isolated, while a basic open set about $\alpha \in S$ is $\{\alpha\}~\cup~$a tail of $\lambda_\alpha$.
\end{defn}

\begin{enumerate}
\item[b)]{\cite{Gruenhage} A ladder system space on a stationary subset of $\omega_1$ has the MOP, and also is normal under MA$_{\omega_1}$.
}
\end{enumerate}
Note the first example is separable, while countable sets have countable closures in the second one.
\begin{enumerate}
\item[c)]{
There is also a version of b) consistent with CH, indeed with $\diamondsuit$.  See \cite{S}, \cite{E}, \cite{LT1}.
}
\end{enumerate}
We shall show that the idea of the proof of the MOP for b) (and hence c)) can be used to establish the MOP for:
\begin{enumerate}
\item[d)]{
The tree topology on a a special Aronszajn tree.  This is known to be non-collectionwise Hausdorff, and to be normal under MA$_{\omega_1}$ \cite{F}.
}
\end{enumerate}
as well as for the space of:
\begin{enumerate}
\item[e)]{
Devlin and Shelah \cite{DS} isolate some points of a special Aronszajn tree and manage to force normality while keeping CH.
}
\end{enumerate}

Generalizing the proof in \cite{Gruenhage} that a ladder system space on a stationary subset of $\omega_1$ has the MOP, we obtain:

\begin{thm}\label{thm12a}
Suppose $X$ is locally compact, locally countable, countable sets have countable closures, and $X = \bigcup_{\gamma < \omega_1}X_\gamma$, where each $X_\gamma$ is countable, $X_\gamma \subsetneq X_{\gamma+1}$, and for $\gamma$ a limit, $X_\gamma = \bigcup_{\alpha < \gamma}X_\alpha$.  Further suppose that for $\gamma$ a limit, for each $x$ in the boundary of $X_\gamma$, there is a compact neighborhood $N(x)$ such that for each $\alpha < \gamma$, $N(x) \cap X_\alpha$ is compact.  Then $X$ has the MOP.
\end{thm}

\begin{proof}
Since countable sets have countable closures, without loss of generality we may assume that $\overline{X}_\alpha \subseteq X_{\alpha+1}$.  Since compact sets are countable, 
\[C = \set{\alpha : x \in X_\alpha \text{ implies } N(x) \subseteq X_\alpha}\]
is closed unbounded.  Since $X$ is first countable, each $X_\alpha$ has a countable base $\mc{B}_\alpha$ of compact sets open in $X_\alpha$.  For $\alpha \in C$, $X_\alpha$ is open, so these sets are open in $X$. 

Let $\mc{A}$ be a moving off collection for $\mc{X}$.  For any $\alpha < \omega_1$, there is a countable ordinal $\delta(\alpha) \geq \alpha$ such that for $B \in \mc{B}_\alpha$,  there is an $A \in \mc{A}$ such that $A \subseteq X_{\delta(\alpha)}$ and $A$ is disjoint from $B$.  Then
\[C' = \set{\alpha \in C : \beta < \alpha \text{ implies } \delta(\beta) < \alpha}\]
is closed unbounded.  Take a strictly increasing sequence $\{\gamma_n\}_{n < \omega}$ in $C'$ and let $\gamma = \sup_n \gamma_n$.  Let $\{B_{m, k} : m < \omega\}$ enumerate the basic compact open sets of $X_{\gamma_k}$.  Note $\bigcup \{B_{m,k}: m,k<\omega\}$ is a basis for $X_\gamma$.  Let $\{x_{\gamma, i} : i < \omega\}$ enumerate $\overline{X}_\gamma - X_\gamma$.  For each $j < \omega$, there is an $A_j \in \mc{A}$ with $A_j \subseteq X_{\gamma_j}$ and $A_j \cap \left(\bigcup_{k < j}A_k \cup \bigcup_{m, k < j}B_{m, k} \cup \bigcup_{n, i < j}(N(x_{\gamma, i}) \cap X_{\gamma_n})\right) = \emptyset$.
Then $\{A_j\}_{j < \omega}$ is locally finite in $X_\gamma$, since each $B_{m, k}$ eventually misses the $A_j$'s.  The $x_{\gamma, i}$'s are then the only possible limits of the $A_j$'s.  But $N(x_{\gamma, i})$ is disjoint from $A_j$ for $j > i$.  Thus the $A_j$'s are locally finite in $X$.  Since the $A_j$'s are also closed disjoint, in fact the collection is discrete.
\end{proof}

Note that by Lemma \ref{lem3}, Theorem \ref{thm12a} does not offer a roadmap for constructing a locally compact normal space with the MOP which is not paracompact.

We note, for future reference, that:

\begin{cor}\label{cor13a}
A countable topological sum of spaces satisfying the hypotheses of Theorem \ref{thm12a} also has the MOP.
\end{cor}

\begin{proof}
The sum also satisfies these hypotheses.
\end{proof}

\begin{thm}\label{thm13new}
Suppose $X$ is locally compact, locally countable, $\abs{X} \geq 2^{\aleph_0}$, and every closed subspace of size $2^{\aleph_0}$ has the MOP.  Then $X$ has the MOP.
\end{thm}

\begin{cor}\label{corCH1}
CH implies if $X$ is locally compact, locally countable, and closed subspaces of size $\aleph_1$ have the MOP, then so does $X$.
\end{cor}

\begin{cor}\label{corCH2}
CH implies if $X$ is locally compact, locally countable, and closed subspaces of size $\aleph_1$ are paracompact, then $X$ has the MOP.
\end{cor}

\begin{cor}\label{corCH3}
CH implies if $X$ is locally compact, locally countable, countable subsets have countable closures, and each closed $Y \subseteq X$ of size $\aleph_1$ satisfies the conditions for $X$ in Theorem \ref{thm12a}, then $X$ has the MOP.
\end{cor}

The first and third corollaries are immediate.  The second is because local compactness is closed-hereditary, and locally compact, paracompact spaces have the MOP.

\begin{proof}[Proof of Theorem \ref{thm13new}]
Let $M$ be a countably closed elementary submodel of size $2^{\aleph_0}$ containing the space $X$ and a moving off collection $\mc{A}$ for it.  By first countability, $X \cap M$ is a closed subspace of $X$, so it will suffice to find a discrete collection $\{A_n\}_{n < \omega}$ included in $\mc{A}$, with each $A_n \subseteq X \cap M$, and $\{A_n\}_{n < \omega}$ discrete in $X \cap M$.  It suffices to show $\mc{A} \cap M$ is moving off for $X \cap M$.   Let $F$ be a compact subspace of $X \cap M$.  Since compact sets are countable and $M$ is countably closed, $F \in M$.  Then, since $M \models \mc{A}$ is moving off, $M \models (\exists A \in \mc{A})(F \cap A = \emptyset)$.  But then there is an $A \in \mc{A} \cap M$ such that $F \cap A = \emptyset$.  
\end{proof}

Our previous counterexamples were not $\aleph_1$-collectionwise Hausdorff; now we can get one that satisfies that property:

\begin{itemize}
\item[f)]
A consistent-with-CH example of a locally compact, normal, $\aleph_1$-collectionwise Hausdorff space with the MOP which is not paracompact.
\end{itemize}

A ladder system space $X$ on a non-reflecting stationary set $E$ of $\omega$-cofinal ordinals in $\omega_2$ is easily seen to be $\aleph_1$-collectionwise Hausdorff, because initial segments of $E$ are non-stationary.  In fact, subspaces of size $\leq \aleph_1$ are paracompact, and hence such small closed ones have the MOP.  $X$ is not paracompact because it is not $\aleph_2$-collectionwise Hausdorff.  Shelah \cite{Shelah2} forced to make $X$ normal, consistent with CH. \hfill$\square$

\begin{itemize}
\item[g)]
A Souslin tree with the usual tree topology is collectionwise normal \cite{Fleissner1980}. It has countable extent but is not Lindel\"of, so is not paracompact. By Theorem \ref{thm12a} it has the MOP.
\end{itemize}

A similar proof of the MOP works for any other $\omega_1$-tree with the tree topology, but normal ones that are not paracompact will not be found in ZFC - see \cite{Fleissner1980}.  Gruenhage \cite{G} proved earlier that any Aronszajn tree
has the MOP.

\section{More results in a model of PFA$(S)[S]$}
There are some easy observations about the MOP in the model of PFA$(S)[S]$ we have mentioned earlier.

\begin{thm}\label{thm10}
In the model of Lemma \ref{lem1}, Theorem \ref{thm7}, etc., locally compact, hereditarily normal, countably tight spaces with the MOP are paracompact.
\end{thm}

\begin{proof}
In a countably tight space, countably compact subspaces are closed \cite{EN}.  Closed subspaces of a space satisfying the MOP also satisfy it.  Perfect pre-images of $\omega_1$ are countably compact but not compact.  Now apply Lemmas \ref{lem1} and \ref{lem4}.
\end{proof}

\begin{cor}\label{cor11}
In this model, first countable hereditarily normal spaces satisfying the MOP are paracompact.
\end{cor}

\begin{proof}
By Theorem \ref{thm10} and Lemma \ref{lem5}.
\end{proof}

\begin{thm}\label{thm12}
In this model, locally compact, normal, countably tight spaces with Lindel\"of number $\leq \aleph_1$ satisfying the MOP are paracompact.
\end{thm}

\begin{proof}
They do not include a perfect pre-image of $\omega_1$, so we can apply Lemma \ref{lem3}.
\end{proof}

\begin{cor}\label{cor13}
In this model, first countable, normal spaces with Lindel\"of number $\leq \aleph_1$ satisfying the MOP are paracompact.
\end{cor}
\begin{proof}
Apply Lemma \ref{lem5} and Theorem \ref{thm12}.
\end{proof}

\begin{cor}
In this model, locally compact, normal, countably tight spaces
satisfying the MOP (in particular, first countable normal spaces satisfying the MOP) are paracompact, provided countable sets have Lindel\"of closures.
\end{cor}
\begin{proof}
In \cite{Tallb} it is shown that in this model,
\begin{lem}
In this model, locally compact normal spaces not including a perfect pre-image of $\omega_1$
are paracompact, provided countable sets have Lindel\"of closures.\hfill$\square$
\end{lem}
\renewcommand{\qedsymbol}{}
\end{proof}

\section{Baire powers of function spaces}
\begin{defn}
A space is \emph{weakly $\alpha$-favorable} \cite{C} if Nonempty has a winning strategy in the Banach-Mazur game.  In that game, players take turns picking an open set included in their opponent's pick. The first player, Empty, wins if, after $\omega$ plays, the intersection of the open sets is empty; otherwise the second player, Nonempty, wins.
\end{defn} 

\begin{lem}[{~\cite{Ma}}]\label{lemAlphaFavour}
A locally compact $X$ is paracompact if and only if $C_k(X)$ is weakly $\alpha$-favorable.
\end{lem}
Galvin and Scheepers \cite{GS} note that White \cite{W} showed that all box powers of weakly $\alpha$-favorable spaces are Baire, and then prove:

\begin{thm}\label{thmGS}
If it is consistent there is a proper class of measurable cardinals, then it is consistent that if all box powers of a space are Baire, then the space is weakly $\alpha$-favorable.
\end{thm}

They then ask whether there are any consistent counterexamples.  Let us consider the particular case of $C_k(X)$ for $X$ locally compact.  Their result then entails:

\begin{cor}
If it is consistent there is a proper class of measurable cardinals, then it is consistent that if all box powers of $C_k(X)$ are Baire, where $X$ is locally compact, then $X$ is paracompact.
\end{cor}

Scheepers pointed out to me that Oxtoby \cite{Ox} proved that any product of Baire spaces with a countable base is Baire, but that a Bernstein set of reals is Baire but not weakly $\alpha$-favorable, so in the Theorem, ordinary powers are not enough.  

In fact, they are not even sufficient for the Corollary.  Example b) is a counterexample:

\begin{thm}
Suppose $X$ satisfies the hypotheses of Theorem 12.  Then arbitrary powers of $C_k(X)$ are Baire.
\end{thm}

\begin{proof}
Fleissner and Kunen \cite{FK} prove

\begin{lem}
Let $\kappa \geq \omega$.  If $X^{\omega}$ is Baire, then $X^{\kappa}$ is Baire.
\end{lem}

McCoy and Ntantu \cite{MN} prove

\begin{lem}
Let $\bigoplus_{\alpha < \lambda}X_\alpha$ be the topological sum of copies of $X$.  Then $C_k\left(\bigoplus_{\alpha < \lambda}X_\alpha\right)$ is homeomorphic to $\left(C_k(X)\right)^{\lambda}$.
\end{lem}

Thus, by Corollary \ref{cor13a}, our assertion that b) is a counterexample is verified.
\end{proof}

The preceding two lemmas prove that:
\begin{thm}
If countable sums of copies of a locally compact $X$ have the MOP, then
arbitrary sums of copies of $X$ have the MOP.
\end{thm}

Surprisingly, there is a consistent example of locally compact spaces $X$ and $Y$, each having the MOP, but $X \oplus Y$ does not have the MOP \cite{Ma}.

We do have one necessity theorem for large cardinals, but do not know whether the hypothesis is vacuous:

\begin{thm}
Suppose that whenever all usual powers of $C_k(X)$ are Baire, for locally compact, $\aleph_1$-collectionwise Hausdorff $X$, then $X$ is paracompact.  Then it is consistent that there is a strong cardinal.
\end{thm}

\begin{proof}
Take a non-reflecting stationary set $E$ of $\omega$-cofinal ordinals in $\lambda^+$, for some $\lambda \geq \mf{c}$.  It is known (attributed to R. Jensen) that if it's consistent no such set exists, then it's consistent there is a strong cardinal - see \cite{J}.  Form a ladder system space $X$ on $E$.  $X$ is not paracompact, but initial segments of it are.  Consider an arbitrary sum $\bigoplus_{\alpha \in S}X_\alpha$ of copies of $X$.   I claim that $\bigoplus_{\alpha \in S}X_\alpha$ has the MOP, whence $[C_k(X)]^{\abs{S}}$ is Baire.  By Corollary \ref{corCH1}, it suffices to show closed subspaces of $\bigoplus_{\alpha \in S}X_\alpha$ of size $\leq 2^{\aleph_0}$ have the MOP.  But they are all paracompact, so they do.  But $X$ is not paracompact.
\end{proof}

A strong cardinal (see \cite{KW} for the definition) has arbitrarily large measurable cardinals below it. Thus, in $V_\kappa$, where $\kappa$ is the least strong cardinal, there is a proper class of measurable cardinals, but no strong cardinal in an inner model. Collapsing these cardinals as in \cite{GS} yields a model in which there is a ladder system space $X$ on a non-reflecting stationary set as above. Some box power of $C_k(X)$ is then not Baire.

Large cardinals can be used to destroy non-reflecting stationary sets; this translates into results about small subspaces being paracompact implying the whole space is paracompact. For example:

\begin{thm}[{~\cite{T2}}]
Martin's Maximum implies that if a first countable space is either generalized ordered or monotonically normal and closed subspaces of size $\aleph_1$ are paracompact, then the space is paracompact.
\end{thm}

For more results of this sort, see \cite{T2}.

In \cite{FleissnerBox}, Fleissner raises the question of whether, if the box product of a
collection of Baire spaces is Baire, its Tychonoff product is Baire.  Also see \cite{FK}.    The converse is not true \cite{FK}.
Note that for box powers, in the model of Galvin and Scheepers, this is true, since
Tychonoff products of weakly $\alpha$-favorable spaces are Baire \cite{W}.
Fleissner also asks whether the box product of Baire spaces with a countable base is
Baire \cite{FleissnerBox}. In this model, this is not true - consider the box powers of a
Bernstein set.

One might be tempted, in view of the countable nature of the Baire Category
Theorem and of weak $\alpha$-favorability, to conjecture that the Baireness of
countable box powers
would consistently be sufficient to imply weak $\alpha$-favorability, at least for
spaces with a countable base. This is not true. L. Zsilinszky \cite{Z} proved:

\begin{thm}
The countable box power of a Baire space with a countable base is Baire.
\end{thm}

But again, a Bernstein set is not weakly $\alpha$-favorable.

\section{Problems}
The most interesting open question in this area is raised in \cite{Gruenhage}:

\begin{prob}
Is there in ZFC a locally compact, normal space with the MOP which is not paracompact (equivalently, $C_k(X)$ is not weakly $\alpha$-favorable)?
\end{prob}

None of the examples we have mentioned exist in the model of PFA$(S)[S]$ we have been using, since in that model there are no Souslin trees, and normal first countable spaces are collectionwise Hausdorff \cite{LT1}.

\begin{prob}[{\cite{GS}}]
Are large cardinals necessary for Theorem \ref{thmGS}?
\end{prob}

\begin{prob}
Can one prove in ZFC that some box power of $C_k$ of a ladder system space on a (non-reflecting?) stationary set of $\omega$-cofinal ordinals is not Baire?
\end{prob}

\begin{prob}[{\cite{FleissnerBox}}]
Can one prove in ZFC that if a box product of a collection of Baire spaces is Baire, then its Tychonoff product is Baire?
\end{prob}
\nocite{*}
\bibliographystyle{acm}
\bibliography{MovingOffNew}



\end{document}